\documentclass[titlepage,a4paper,11pt]{amsart} %automatically loads amsmath and amsthm
\usepackage[foot]{amsaddr}
\usepackage{amssymb}
\usepackage[lite]{amsrefs}
\usepackage{t1enc}
\usepackage[latin1]{inputenc}
\usepackage[english]{babel}
\usepackage{mathrsfs} %for mathscr
\usepackage{hyperref}
\usepackage[all]{xy} % for diagrams
\usepackage[lite]{amsrefs}
\usepackage[left=3cm,top=3cm,right=3cm,bottom=3cm]{geometry}

\newtheorem{thm}{Theorem}[section]
\newtheorem*{thm*}{Theorem}
\newtheorem{prop}[thm]{Proposition}
\newtheorem{lem}[thm]{Lemma}
\newtheorem{cor}[thm]{Corollary}
\newtheorem{defn}[thm]{Definition}
\newtheorem{rem}[thm]{Remark}

\newenvironment{hproof}{%
  \proof}{\endproof}

\DeclareMathOperator\Le{L}

\DeclareMathOperator\spec{spec}
\DeclareMathOperator\Hi{\mathcal H}
\DeclareMathOperator\B{\mathcal B}

\DeclareMathOperator\cen{c}
\DeclareMathOperator\C{\mathbb{C}}
\DeclareMathOperator\R{\mathbb{R}}

\DeclareMathOperator\s{\mathbb{S}}

\DeclareMathOperator\id{\mathbf{1}}

\hypersetup{
  colorlinks   = true, %Colours links instead of ugly boxes
  urlcolor     = blue, %Colour for external hyperlinks
  linkcolor    = blue, %Colour of internal links
  citecolor   = red    %Colour of citations
}
\title[Geometry of spectral bounds of curves of unitary operators]{Geometry of spectral bounds of curves of unitary operators}
\date{\today}
\author{Martin Miglioli}
\email[Martin Miglioli]{martin.miglioli@gmail.com}
\address[Martin Miglioli]{Instituto Argentino de Matem\'atica-CONICET. Saavedra 15, Piso 3, (1083) Buenos Aires, Argentina}
\thanks{The author was supported by IAM-CONICET, grants PIP 2010-0757 (CONICET) and PICT 2010-2478 (ANPCyT)}

\begin{document}
\begin{abstract}
This article presents a new proof of a theorem concerning bounds of the spectrum of the product of unitary operators and a generalization for differentiable curves of this theorem.  The proofs involve metric geometric arguments in the group of unitary operators and the sphere where these operators act.  \\

\medskip

\end{abstract}

\maketitle
%\tableofcontents

%\section*{COMMENTS - TO BE DELETED FOR FINAL DRAFT}

%format:
%\begin{itemize}
%\end{itemize}

%content:
%\begin{itemize}
%\end{itemize}

\section{Introduction}
In this article we present metric proofs of facts about the spectrum of curves of unitary operators. In particular, if $\theta_+(u)$ denotes the maximum argument of the spectrum of a unitary operator $u$, then $\theta_+(uv)\leq\theta_+(u)+\theta_+(v)$ for unitary operators $u$ and $v$ such that $-1\notin\spec(u)$, $-1\notin\spec(v)$ and $\theta_+(u)+\theta_+(v)<\pi$. An analogous statement holds for the minimum argument of unitary operators. This result was proved previously with different techniques in \cite{nudelman,kato,yaf,aw,chau0,chau,childs}, see the paragraph after Theorem (1) in \cite{chau} for a discussion of the previous literature. The bound on the spectrum also follows from Thompson's theorem \cite{thompson} which uses the solution to Horn's problem. When the inequalities are equalities, there is a relation between the eigenspaces of the eigenvalues with maximum and minimum arguments of the the unitaries and its product, see \cite{chau}. 

In the present paper we show them as consequences of the triangle inequality in the group of unitary operators and in the unit sphere where the unitaries act. The proofs apply directly to the infinite dimensional context without any approximation argument. We also prove a generalization of these results for ``infinite products of unitaries'', that is, for piecewise $C^1$ curves. In Section \ref{prel} we review the results used in this article, in Section \ref{maintheoremtwo} we present the geometric proof of the results mentioned above, and in Section \ref{maintheoremcurve} we present the generalization for piecewise $C^1$ curves. We prove the special case first since there is interest in the literature for alternative proofs of this theorem, also because it helps to understand the statement and the proof of the generalization.     

\section{Preliminaries}\label{prel}

In this section we recall results about the metric geometry of spaces of unitaries endowed with the bi-invariant Finsler metric derived from the uniform norm. See Section 2 and Section 4 of \cite{larotonda} for a study of Banach and Frechet Lie groups endowed with a bi-invariant Finlser metric. 

Let $\Hi$ be a separable Hilbert space and let $U$ be the group of unitaries. On the algebra of bounded operators $\B(\Hi)$ the operator norm of $x\in\B(\Hi)$ is $\|x\|=\sup_{\xi\in \Hi}\frac{\|x\xi\|}{\|\xi\|}$. This norm is invariant by conjugation by unitaries, so by right or left translation we can define a norm on the tangent spaces at all points of $U$. With this Finsler structure we define a metric on $U$. Let $\Le$ denote the length functional for piecewise $C^1$ curves $\alpha:[a,b]\to U$ measured with the $\|\cdot\|$ norm
$$\Le(\alpha)=\int_{a}^{b}\|\dot{\alpha}_t\|dt=\int_{a}^{b}\|\alpha_t^{-1}\dot{\alpha}_t\|dt=\int_{a}^{b}\|\dot{\alpha}_t\alpha_t^{-1}\|dt,$$
where $\alpha_t^{-1}\dot{\alpha}_t$ and $\dot{\alpha}_t\alpha_t^{-1}$ are the left and right logarithmic derivatives of $\alpha$, and the derivative of a curve $\alpha$ is denoted with $\dot{\alpha}$. The rectifiable distance between $u$ and $v$ in $U$ is given by $$d(u,v)= \inf\{\Le(\gamma):\gamma \subseteq U \mbox{ joins } u \mbox{ and }v\}.$$ 
This metric is invariant by left and right translation, that is, $d(uv,uw)=d(v,w)$ and $d(vu,wu)=d(v,w)$ for $u,v,w\in U$. 

We next recall Proposition 5.2. of \cite{esteban} and its proof for the convenience of the reader. This result was proved in  \cite{portarecht,atkin1,atkin2}, in a much more general setting it was proved in \cite[Theorem 4.11]{larotonda}. Let $\s=\{\xi\in\Hi:\|\xi\|=1\}$ stand for the unit sphere in $\Hi$ endowed with its canonical Hilbert-Riemann metric and distance function $d_{\s}$. Let $\B(\Hi)_h$ be the hermitian operators in $\B(\Hi)$ and note that the Lie algebra of $U$ is the space $i\B(\Hi)_h$ of skew hermitian operators. For $\xi\in\Hi$ we define the smooth map
$$\rho_\xi:U\to \s,\qquad\rho_\xi(u)=u\xi.$$

\begin{lem}\label{decreaselength}
Let $\xi\in\Hi$ and let $\gamma:[0,1]\to U$ be a piecewise $C^1$ curve joining $v$ and $w$, then $\Le(\rho_\xi(\gamma))\leq \Le(\gamma)$, where $\Le$ denotes the length functional of curves. If $v,w\in U$ then 
$$d_{\s}(v\xi,w\xi)\leq d(v,w).$$  
\end{lem}

\begin{proof}
The differential of the map $\rho_\xi$ at an $u\in U$ is 
$$d(\rho_\xi)_u:iu\B(\Hi)_h\to T_{u\xi}\s,\qquad d(\rho_\xi)_u(iux)=iux\xi,$$
for $x\in\B(\Hi)_h$. Since $\|d(\rho_\xi)_u(iux)\|=\|iux\xi\|\leq\|ux\|$ this map is norm decreasing. Hence, if $\gamma$ is a piecewise smooth curve in $U$ then $\Le(\rho_\xi(\gamma))\leq \Le(\gamma)$. By taking the infimum over curves in $U$ joining $v$ and $w$ we get $d_{\s}(v\xi,w\xi)\leq d(v,w)$. 
\end{proof}

\begin{prop}\label{minimalgeod}
Let $u\in U$ and $x\in\B(\Hi)_{h}$ with $\|x\|\leq\pi$. Then the smooth curve $\mu(t)= ue^{itx}$ has minimal length along its path, for all $t\in[-1,1]$. Any pair of unitaries $u,v\in U$ can be joined by such a curve.
\end{prop}

\begin{hproof}
Suppose that $x$ has a norming eigenvector, that is, a vector $\xi\in\s$ such that $x\xi=\lambda \xi$ with $\lambda=\pm \|x\|$. Consider the curve in $\mu(t)=ue^{itx}$ in $U$. We have 
$$\rho_\xi(\mu(t))=ue^{itx}\xi=e^{it\lambda}u\xi,$$
and since $\lambda\leq\pi$ this curve is minimal in $\s$, hence $\Le(\rho_\xi(\mu))=d_{\s}(u\xi,ue^{ix}\xi)$. Also $\frac{d}{dt}(\rho_\xi(\mu(t))=i\lambda e^{it\lambda}u\xi$, therefore 
$$\left\lVert\frac{d}{dt}(\rho_\xi(\mu(t))\right\rVert=\vert \lambda\vert=\|x\|=\left\lVert\frac{d}{dt}\mu(t)\right\rVert,$$
and we conclude that $\Le(\mu)=\Le(\rho_\xi(\mu))$. By Lemma \ref{decreaselength} we have $d_{\s}(u\xi,ue^{ix}\xi)\leq d(u,ue^{ix})\leq \Le(\gamma)$ for all $\gamma$ in $U$ joining $u$ and $ue^{ix}$. The general case follows by approximating an arbitrary $x\in\B(\Hi)_h$ with operators that have a norming eigenvector. Any unitary operator can be written in the form $e^{ix}$ for some self adjoint operator $x$ with $\|x\|\leq\pi$. If $u,v\in U$, there exists such $x$ satisfying $u^{-1}v = e^{ix}$. Then $\mu(t)=ue^{itx}$ is a minimal curve in $U$ joining $\mu(0)=u$ and $\mu(1)=ue^{ix}=v$.
\end{hproof}

The exponential map is injective on $B_\pi=\{x\in i\B(\Hi)_h:\|x\|<\pi\}$, that is, the injectivity radius of the exponential map in $U$ is $\pi$. We define $V_\pi=\exp(B_\pi)$. The next result is \cite[Theorem 4.22]{larotonda} stated in the particular context of the group $U$, see Section 5.1.1 of the same article. We denote with $\id$ the identity operator in $\Hi$.

\begin{thm}\label{normingfunc}
Let $z\in  i\B(\Hi)_h$ with $\|z\|<\pi$, let $\gamma:[a,b]\to U$ be a piecewise $C^1$ curve joining $\id$ and $e^z$ in $U$. The following are equivalent
\begin{itemize}
\item $\gamma$ is a short curve in $U$, that is, $\Le(\gamma)=d(\id,e^z)=\|z\|$.
\item $\gamma=e^\Gamma \subseteq V_\pi$ and for any norming functional $\psi$ of $z$, $\psi(\dot{\Gamma}_t)=\|\gamma_t^{-1}\dot{\gamma}_t\|$ for all $t\in [a,b]$.
\item $\gamma=e^\Gamma$ for $\Gamma \subseteq B_\pi$ and there exists a unit norm functional $\psi$ such that $\psi(\Gamma_t)=\|\Gamma_t\|$ and $\psi(\gamma_t^{-1}\dot{\gamma}_t)=\|\gamma_t^{-1}\dot{\gamma}_t\|$ for all $t\in [a,b]$. Thus $z/{\|z\|}$ and $\gamma^{-1}\dot{\gamma}/{\|\gamma^{-1}\dot{\gamma}\|}$ sit inside a face of the unit sphere of $i\B(\Hi)_h$ endowed with the norm $\|\cdot\|$.
\end{itemize}
\end{thm}

\begin{rem}
The proof of Theorem \ref{normingfunc} uses a generalization in the case of groups endowed with a conjugation invariant Finsler metric of the Gauss' Lemma of Riemannian geometry: the differential of the exponential map along a geodesic preserves angles with the geodesic speed vector, see \cite[Lemma 4.5]{larotonda}. This lemma asserts that for $x,y\in i\B(Hi)_h$ and $\psi$ a unit norm norming functional for $y$, that is, $\psi(y)=\|y\|$ and $\|\psi\|=1$, the following equality holds
$$\psi(e^{-y}\exp_{*y}x)=\psi(x),$$
where $\exp_{*y}$ denotes the differential of the exponential map $\exp$ at $y\in i\B(Hi)_h$.
\end{rem}

\section{Spectral bound for the product of unitaries}\label{maintheoremtwo} 
In this section we prove spectral bounds for the product of two unitary operators using the geometry of $(U,d)$. In the case of equality in the bounds we study the relation among eigenspaces associated to the eigenvalues with maximum and minimum arguments of the unitaries. This is done using the action of $U$ on $\s$.

\begin{defn}
For a unitary operator $u$ such that $-1\notin\spec(u)$ we define
$$\theta_+(u)=\max\spec(-i\log(u))\quad\mbox{ and }\quad \theta_-(u)=\min\spec(-i\log(u)),$$
where $\log$ is the principal branch of the logarithm and $\spec$ is the spectrum of an operator. We also define
$$\Hi_+(u)=\{\xi\in\Hi:u\xi=e^{i\theta_+(u)}\xi\}\quad\mbox{ and }\quad\Hi_-(u)=\{\xi\in\Hi:u\xi=e^{i\theta_-(u)}\xi\}$$
as the eigenspaces of $e^{i\theta_+(u)}$ and $e^{i\theta_-(u)}$ respectively.
\end{defn}

\begin{lem}\label{lempspec}
If $u\in U$ such $d(\id,u)<\pi$ then $-1\notin\spec(u)$. If $-1\notin\spec(u)$ then
$$d(\id,u)=\max\{\theta_+(u),-\theta_-(u)\}.$$
Also, if $-\pi-\theta_-(u)<\lambda<\pi-\theta_+(u)$ then $-1\notin\spec(e^{i\lambda}u)$,
$$\theta_+(e^{i\lambda}u)=\lambda +\theta_+(u) \quad\mbox{ and }\quad\theta_-(e^{i\lambda}u)=\lambda +\theta_-(u).$$
\end{lem}

\begin{proof}
By Proposition \ref{minimalgeod} the curve $\mu(t)=e^{t\log(u)}$ joins the identity operator $\id$ and $u$ and has minimal length $d(\id,u)=\Le(\mu)=\|\log(u)\|$. If $\|\log(u)\|<\pi$ then $-1\notin\spec(u)$, and in this case  $\|\log(u)\|=\max\{\theta_+(u),-\theta_-(u)\}$. The second assertion is straightforward and left to the reader. 
\end{proof}

\begin{thm}\label{ineqspec}
Let $u$ and $v$ be unitary operators acting on the Hilbert space $\Hi$ such that $-1\notin\spec(u)$ and $-1\notin\spec(v)$, and such that $\theta_+(u)+\theta_+(v)<\pi$ and $\theta_-(u)+\theta_-(v)>-\pi$. Then $-1\notin\spec(uv)$,
$$\theta_+(u)+\theta_+(v)\geq\theta_+(uv)\quad\mbox{ and }\quad\theta_-(u)+\theta_-(v)\leq\theta_-(uv).$$
If $\theta_+(uv)=\theta_+(u)+\theta_+(v)<\pi$, then
$$\Hi_+(u)\cap\Hi_+(v)=\Hi_+(uv),$$
and if $\theta_-(uv)=\theta_-(u)+\theta_-(v)>-\pi$, then
$$\Hi_-(u)\cap\Hi_-(v)=\Hi_-(uv).$$
\end{thm}

\begin{proof}
We define $u_{\cen}=e^{-i\frac{1}{2}(\theta_+(u)+\theta_-(u))}u$ and $v_{\cen}=e^{-i\frac{1}{2}(\theta_+(v)+\theta_-(v))}v$. By Lemma \ref{lempspec} these operators satisfy $\theta_+(u_{\cen})=-\theta_-(u_{\cen})=\frac{1}{2}(\theta_+(u)-\theta_-(u))$ and $\theta_+(v_{\cen})=-\theta_-(v_{\cen})=\frac{1}{2}(\theta_+(v)-\theta_-(v))$, hence 

$$d(\id,u_{\cen})=\frac{1}{2}(\theta_+(u)-\theta_-(u))\quad \mbox{ and }\quad   d(\id,v_{\cen})=\frac{1}{2}(\theta_+(v)-\theta_-(v)).$$
By the triangle inequality in $(U,d)$ we have
\begin{align} 
d(\id,u_{\cen} v_{\cen})&\leq d(\id,u_{\cen})+d(u_{\cen},u_{\cen} v_{\cen})\nonumber\\
&=d(\id,u_{\cen})+d(\id, v_{\cen})\label{ineq1}\\
&=\frac{1}{2}(\theta_+(u)-\theta_-(u))+\frac{1}{2}(\theta_+(v)-\theta_-(v))<\pi,\nonumber
\end{align}
where $d(u_{\cen},u_{\cen} v_{\cen})=d(\id,v_{\cen})$ follows from translation invariance of the metric. Since $d(\id,u_{\cen} v_{\cen})<\pi$ by Lemma \ref{lempspec} we conclude that $-\id\notin u_{\cen} v_{\cen}$, so that $\theta_+(u_{\cen} v_{\cen})$ is well defined and   
\begin{align}\label{ineq2} 
d(\id,u_{\cen} v_{\cen})\geq \theta_+(u_{\cen} v_{\cen}).
\end{align}
Note that $e^{i\frac{1}{2}S}u_{\cen} v_{\cen}=uv$ where $S=\theta_+(u)+\theta_-(u)+\theta_+(v)+\theta_-(v)$, and since $-\pi<\theta_-(u)+\theta_-(v)\leq\theta_+(u)+\theta_+(v)<\pi$ we have $-\pi-\theta_-(u_{\cen}v_{\cen})<\frac{1}{2}S<\pi-\theta_+(u_{\cen}v_{\cen})$, hence by Lemma \ref{lempspec} $-\id\notin\spec(uv)$ and $\theta_+(u_{\cen} v_{\cen})+\frac{1}{2}S=\theta_+(u v)$. If we combine inequalities (\ref{ineq1}) and (\ref{ineq2}) we get
$$\frac{1}{2}(\theta_+(u)-\theta_-(u)) +\frac{1}{2}(\theta_+(v)-\theta_-(v))\geq \theta_+(uv)-\frac{1}{2}(\theta_+(u)+\theta_-(u)+\theta_+(v)+\theta_-(v)),$$
which is equivalent to $\theta_+(u)+\theta_+(v)\geq\theta_+(uv)$. 

In the same way, since $d(\id,u_{\cen} v_{\cen})\geq - \theta_
-(u_{\cen} v_{\cen})=-(\theta_-(u v)-\frac{1}{2}S)$ we get
$$\frac{1}{2}(\theta_+(u)-\theta_-(u)) +\frac{1}{2}(\theta_+(v)-\theta_-(v))\geq -\theta_-(uv)+\frac{1}{2}(\theta_+(u)+\theta_-(u)+\theta_+(v)+\theta_-(v)),$$
which is equivalent to $\theta_-(u)+\theta_-(v)\leq\theta_-(uv)$.

Assume that $\theta_+(uv)=\theta_+(u)+\theta_+(v)<\pi$. The proof of the inclusion $\Hi_+(u)\cap\Hi_+(v)\subseteq\Hi_+(uv)$ is straightforward and left to the reader. We prove that $\Hi_+(uv)\subseteq\Hi_+(u)\cap\Hi_+(v)$. Note that $\Hi_+(u_{\cen})=\Hi_+(u)$, $\Hi_+(v_{\cen})=\Hi_+(v)$ and $\Hi_+(u_{\cen}v_{\cen})=\Hi_+(uv)$, so it is enough to prove the statement for $u_{\cen}$ and $v_{\cen}$. The following inequality holds
$$\theta_+(u_{\cen}v_{\cen})=\theta_+(u_{\cen})+\theta_+(v_{\cen})=-\theta_-(u_{\cen})-\theta_-(v_{\cen})\geq -\theta_-(u_{\cen}v_{\cen}),$$
hence by Lemma \ref{lempspec}, $d(\id,u_{\cen}v_{\cen})=\theta_+(u_{\cen}v_{\cen})<\pi$. Let $\xi\in\Hi_+(u_{\cen}v_{\cen})\cap\s$, so that 
$$u_{\cen}v_{\cen}\xi=e^{i\theta_+(u_{\cen}v_{\cen})}\xi\quad\mbox{ and }\quad d_{\s}(\xi,u_{\cen}v_{\cen}\xi)=\theta_+(u_{\cen}v_{\cen}),$$ 
since $\theta_+(u_{\cen}v_{\cen})<\pi$. By Lemma \ref{decreaselength} we conclude that 
$$d_{\s}(\xi,v_{\cen}\xi)\leq d(\id,v_{\cen})=\theta_+(v_{\cen})\quad\mbox{ and }\quad d_{\s}(v_{\cen}\xi,u_{\cen}v_{\cen}\xi)\leq d(\id,u_{\cen})=\theta_+(u_{\cen}).$$
Therefore
\begin{align*}
\theta_+(u_{\cen}v_{\cen})&=d_{\s}(\xi,u_{\cen}v_{\cen}\xi)\leq d_{\s}(\xi,v_{\cen}\xi)+d_{\s}(v_{\cen}\xi,u_{\cen}v_{\cen}\xi)\\
&\leq \theta_+(u_{\cen})+\theta_+(w_{\cen})=\theta_+(u_{\cen}v_{\cen}),
\end{align*}
and this implies 
$$d_{\s}(\xi,u_{\cen}v_{\cen}\xi)= d_{\s}(\xi,v_{\cen}\xi)+d_{\s}(v_{\cen}\xi,u_{\cen}v_{\cen}\xi),$$ 
$d_{\s}(\xi,v_{\cen}\xi)=\theta_+(v_{\cen}) $ and $d_{\s}(v_{\cen}\xi,u_{\cen}v_{\cen}\xi)=\theta_+(u_{\cen})$. The triangle inequality in $\s$ is an equality and this means that $\xi$, $v_{\cen}\xi$ and $u_{\cen}v_{\cen}\xi=e^{i\theta_+(u_{\cen}v_{\cen})}\xi$ all lie in the same geodesic in $\s$. Therefore $v_{\cen}\xi=e^{id_{\s}(\xi,v_{\cen}\xi)}\xi=e^{i\theta_+(v_{\cen})}\xi$, that is $\xi\in \Hi_+(v_{\cen})$. Since $e^{i\theta_+(u_{\cen}v_{\cen})}\xi=u_{\cen}v_{\cen}\xi=e^{i\theta_+(v_{\cen})}u_{\cen}\xi$ we see that $u_{\cen}\xi=e^{i\theta_+(u_{\cen}v_{\cen})-i\theta_+(v_{\cen})}\xi=e^{i\theta_+(u_{\cen})}\xi$, so that $\xi\in \Hi_+(u_{\cen})$.  The proof of $\Hi_+(u_{\cen}v_{\cen})\subseteq\Hi_+(u_{\cen})\cap\Hi_+(v_{\cen})$ is complete. The proof of last equality in the statement of the theorem is analogous. 

\end{proof}

\begin{rem}
Using Lemma \ref{lempspec} one can change one of the two inequalities $\theta_+(u)+\theta_+(v)<\pi$ and $\theta_-(u)+\theta_-(v)>-\pi$ in the statement of Theorem \ref{ineqspec} into an inequality with $\leq$ or $\geq$ respectively. 
\end{rem}

\bigskip

\section{Spectral bounds for curves of unitaries}\label{maintheoremcurve}
In this section we generalize Theorem \ref{ineqspec} to the case of piecewise $C^1$ curves in $U$. The second part of the proof of the main theorem of the section has two proofs, one is a generalization of the argument in Theorem \ref{ineqspec} and the other one uses Theorem \ref{normingfunc}, which was proved in \cite{larotonda}.

\begin{defn}
For a skew hermitian operator $x$ we define
$$\phi_+(x)=\max\spec(-ix)\quad\mbox{ and }\quad \phi_-(x)=\min\spec(-ix).$$
The operator $x_{\cen}$ with centred spectrum is 
$$x_{\cen}=x-\frac{i}{2}(\phi_+(x)+\phi_-(x)).$$
We also define
$$\Hi_+(u)=\{\xi\in\Hi:x\xi=i\phi_+(x)\xi\}\quad\mbox{ and }\quad\Hi_-(x)=\{\xi\in\Hi:x\xi=i\phi_-(x)\xi\}$$
as the eigenspaces of $i\phi_+(x)$ and $i\phi_-(x)$ respectively.
\end{defn}

\begin{rem}
Note that if $x\in i\B(\Hi)_h$ satisfies $\|x\|<\pi$, then $\phi_+(x)=\theta_+(e^x)$, $\phi_-(x)=\theta_-(e^x)$, $\Hi_+(x)=\Hi_+(e^x)$ and $\Hi_-(x)=\Hi_-(e^x)$.
\end{rem}

\begin{lem}\label{lrlog}
Let $\gamma:[a,b]\to U$ be a piecewise $C^1$ curve such that $\gamma_a=\id$, and let $\xi\in \Hi$. If 
$$\gamma_t^{-1}\dot{\gamma}_t\xi=if(t)\xi$$
for a piecewise continuous function $f:[a,b]\to\R$ and $t\in [a,b]$ then 
$$\dot{\gamma}_t\gamma_t^{-1}\xi=\gamma_t^{-1}\dot{\gamma}_t\xi=e^{i\int_a^t f(s)ds}\xi$$
for $t\in [a,b]$. The same conclusion holds if $\dot{\gamma}_t\gamma_t^{-1}\xi=if(t)\xi$ for $t\in [a,b]$.
\end{lem}

\begin{proof}
Since $\gamma_t^{-1}\dot{\gamma}_t\xi=if(t)\xi$ we have
$$\dot{\gamma}_t\xi=if(t)\gamma_t\xi$$
for all $t\in [a,b]$. If we set $\eta_t=\gamma_t\xi$ we see that $\eta_a=\id\xi=\xi$ and $\dot{\eta}_t=if(t)\eta_t$ for all $t\in [a,b]$. Therefore 
$$\gamma_t\xi=\eta_t=e^{i\int_a^tf(s)ds}\xi$$
for all $t\in [a,b]$, and from this formula the conclusion of the lemma follows.
\end{proof}

\begin{thm}\label{mainthm}
Let $\gamma:[a,b]\to U$ be a piecewise $C^1$ curve such that $\gamma_a=\id$,
$$\int_a^b\phi_+(\dot{\gamma}_t\gamma_t^{-1})dt<\pi\quad\mbox{ and }\quad\int_a^b\phi_-(\dot{\gamma}_t\gamma_t^{-1})dt>-\pi,$$
then $-1\notin\spec(\gamma_b)$ and
$$\theta_+(\gamma_b)\leq \int_a^b\phi_+(\dot{\gamma}_t\gamma_t^{-1})dt\quad\mbox{ and }\quad\theta_-(\gamma_b)\geq \int_a^b\phi_-(\dot{\gamma}_t\gamma_t^{-1})dt.$$
If $\theta_+(\gamma_b)=\int_a^b\phi_+(\dot{\gamma}_t\gamma_t^{-1})dt$ then
$$\Hi_+(\gamma_b)=\bigcap_{t\in [a,b]}\Hi_+(\dot{\gamma}_t\gamma_t^{-1})=\bigcap_{t\in [a,b]}\Hi_+(\gamma_t^{-1}\dot{\gamma}_t),$$
and if $\theta_-(\gamma_b)=\int_a^b\phi_-(\dot{\gamma}_t\gamma_t^{-1})dt$ then
$$\Hi_-(\gamma_b)=\bigcap_{t\in[a,b]}\Hi_-(\dot{\gamma}_t\gamma_t^{-1})=\bigcap_{t\in[a,b]}\Hi_-(\gamma_t^{-1}\dot{\gamma}_t),$$
\end{thm}

\begin{proof}
Define the function $S:[a,b]\to\R$ by
$$S_t=\int_a^t(\phi_+(\dot{\gamma}_s\gamma_s^{-1})+\phi_-(\dot{\gamma}_s\gamma_s^{-1}))ds,$$
and the curve $\alpha:[a,b]\to U$ by
$$\alpha_t=\gamma_te^{-\frac{i}{2}S_t}.$$ 
Note that $\alpha_a=\id$ and the right logarithmic derivative of $\alpha$ is
\begin{align*}
\dot{\alpha}_t\alpha_t^{-1}&=(\dot{\gamma}_te^{-\frac{i}{2}S_t}-\frac{i}{2}(\phi_+(\dot{\gamma}_t\gamma_t^{-1})+\phi_-(\dot{\gamma}_t\gamma_t^{-1}))\gamma_te^{-\frac{i}{2}S_t})\gamma_t^{-1}e^{\frac{i}{2}S_t}\\
&=\dot{\gamma}_t\gamma_t^{-1}-\frac{i}{2}(\phi_+(\dot{\gamma}_t\gamma_t^{-1})+\phi_-(\dot{\gamma}_t\gamma_t^{-1}))=(\dot{\gamma}_t\gamma_t^{-1})_{\cen},
\end{align*}
which is the centred right logarithmic derivative of $\gamma$. By the definition of distance in $U$ we get
\begin{align}\label{e1}
d(\id,\alpha_b)&\leq \int_a^b\|\dot{\alpha}_t\alpha_t^{-1}\|dt=\int_a^b\|(\dot{\gamma}_t\gamma_t^{-1})_{\cen}\|dt\\
&=\frac{1}{2}\int_a^b(\phi_+(\dot{\gamma}_t\gamma_t^{-1})-\phi_-(\dot{\gamma}_t\gamma_t^{-1}))dt<\pi,
\end{align}
$\int_a^b\phi_+(\dot{\gamma}_t\gamma_t^{-1})dt<\pi$ and $\int_a^b\phi_-(\dot{\gamma}_t\gamma_t^{-1})dt>-\pi$. Therefore $-1\notin\spec(\alpha_b)$ and $d(\id,\alpha_b)=\max\{\theta_+(\alpha_b),-\theta_-(\alpha_b)\}$. In order to apply Lemma \ref{lempspec} to $e^{\frac{i}{2}S_b}\alpha_b=\gamma_b$ we have to check that 
$$-\pi-\theta_-(\alpha_b)<\frac{1}{2}S_b<\pi-\theta_+(\alpha_b).$$
We verify the second inequality, the first inequality is proved analogously. Note that 
$$\theta_+(\alpha_b)\leq d(\id,\alpha_b)\leq \frac{1}{2}\int_a^b(\phi_+(\dot{\gamma}_t\gamma_t^{-1})-\phi_-(\dot{\gamma}_t\gamma_t^{-1}))dt,$$
where the second inequality is (\ref{e1}). Hence
$$\theta_+(\alpha_b)+\frac{1}{2}S_b\leq \int_a^b\phi_+(\dot{\gamma}_t\gamma_t^{-1})dt<\pi.$$
So by Lemma \ref{lempspec}
$$\theta_+(\gamma_b)=\theta_+(\alpha_b)+\frac{1}{2}S_b\leq \int_a^b\phi_+(\dot{\gamma}_t\gamma_t^{-1})dt<\pi$$
and
$$\theta_-(\gamma_b)=\theta_-(\alpha_b)+\frac{1}{2}S_b \geq \int_a^b\phi_-(\dot{\gamma}_t\gamma_t^{-1})dt>-\pi.$$
This proves the first assertion of the theorem. 

We now prove the assertion about the equality cases in two different ways. Note that 
$\Hi_+(\gamma_b)=\Hi_+(\alpha_b)$ and 
$$\Hi_+(\dot{\gamma}_t\gamma_t^{-1})=\Hi_+((\dot{\gamma}_t\gamma_t^{-1})_{\cen})=\Hi_+(\dot{\alpha}_t\alpha_t^{-1})$$
for $t\in [a,b]$.
Assume that 
$$\theta_+(\gamma_b)=\int_a^b\phi_+(\dot{\gamma}_t\gamma_t^{-1})dt,$$
then 
\begin{align}\label{eq0}
\theta_+(\alpha_b)=\theta_+(\gamma_b)-\frac{1}{2}S_b=\int_a^b(\phi_+(\dot{\gamma}_t\gamma_t^{-1})-\phi_-(\dot{\gamma}_t\gamma_t^{-1}))dt=\int_a^b\|\dot{\alpha}_t\alpha_t^{-1}\|dt.
\end{align}
From Lemma \ref{lrlog} it follows that
$$\bigcap_{t\in [a,b]}\Hi_+(\dot{\gamma}_t\gamma_t^{-1})\subseteq\Hi_+(\gamma_b)\quad\mbox{ and }\quad\bigcap_{t\in [a,b]}\Hi_+(\gamma_t^{-1}\dot{\gamma}_t)\subseteq\Hi_+(\gamma_b).$$

The first proof of the opposite inclusions is as follows. Let $\xi\in\Hi_+(\alpha_b)\cap\s$ and define $\rho_\xi:U\to\s$ as before. Note that $\rho_\xi(\alpha)$ is a curve in $\s$ joining $\alpha_a\xi=\xi$ and $\alpha_b\xi=e^{i\theta_+(\alpha_b)}\xi$. Since $\theta_+(\alpha_b)<\pi$ we have
\begin{align}\label{eq2}
\theta_+(\alpha_b)=d_{\s}(\xi,e^{i\theta_+(\alpha_b)}\xi)\leq \Le(\rho_\xi(\alpha))\leq \Le(\alpha)= \int_a^b\|\dot{\alpha}_t\alpha_t^{-1}\|dt=\theta_+(\alpha_b),
\end{align}
where the first inequality follows from the definition of $d_{\s}$ and the second inequality follows from the length decreasing property of Lemma \ref{decreaselength}. Hence $\Le(\rho_\xi(\alpha))=\Le(\alpha)$, that is 
$$\int_a^b\|\dot{\alpha}_t\xi\|dt=\int_a^b\|\dot{\alpha}_t\alpha_t^{-1}\|dt=\int_a^b\|\dot{\alpha}_t\|dt.$$
Since $\|\dot{\alpha}_t\xi\|\leq \|\dot{\alpha}_t\|$ for all $t\in [a,b]$ this equality of lengths of curves implies that
\begin{align}\label{eq1}
\|\dot{\alpha}_t\xi\|= \|\dot{\alpha}_t\|=\|\dot{\alpha}_t\alpha_t^{-1}\|
\end{align}
for all $t\in [a,b]$. Equation (\ref{eq2}) states that $d_{\s}(\xi,e^{i\theta_+(\alpha_b)}\xi)= \Le(\rho_\xi(\alpha))$, and since $\theta_+(\alpha_b)<\pi$ we conclude that $\rho_\xià(\alpha)$ is a geodesic in $\s$ joining $\xi$ and $e^{i\theta_+(\alpha_b)}\xi$. Hence
$$\rho_\xi(\alpha_t)=\alpha_t\xi=e^{if(t)}\xi$$
for a non decreasing piecewise $C^1$ function $f:[a,b]\to\R$ such that $f(a)=0$ and $f(b)=\theta_+(\alpha_b)$. Observe that 
$$\dot{\alpha}_t\xi=if'(t)e^{if(t)}\xi \quad\mbox{  and  }\quad\xi=e^{if(t)}\alpha^{-1}_t\xi$$
for all $t\in [a,b]$, therefore
$$-i \dot{\alpha}_t\alpha_t^{-1}\xi=f'(t)\xi.$$
By equation (\ref{eq1}) 
$$\|\dot{\alpha}_t\xi\|=f'(t)=\|\dot{\alpha}_t\alpha_t^{-1}\|$$
for all $t\in [a,b]$, hence $\xi\in\Hi_+(\dot{\alpha}_t\alpha_t^{-1})$ for all $t\in [a,b]$. By Lemma \ref{lrlog} the same assertion holds for the left logarithmic derivatives. The proof for the minimum of the spectrum is similar and we omit it.  

We now turn to the second proof of the last part of the theorem.  Since
$$\theta_+(\alpha_b)=\int_a^b\|\dot{\alpha}_t\alpha_t^{-1}\|dt=\int_a^b\|\alpha_t^{-1}\dot{\alpha}_t\|dt=\Le(\alpha)\geq d(\id,\alpha_b)$$
by equation (\ref{eq0}) and $\theta_+(\alpha_b)\leq d(\id,\alpha_b)$ we see that $\Le(\alpha)= d(\id,\alpha_b)$. This means that $\alpha$ is a geodesic joining $\id$ and $\alpha_b$. If $\xi\in\Hi_+(\alpha_b)\cap\s$ we define a unit norm functional $\psi_\xi:i\B(\Hi)_h\to \R$ on the Lie algebra $i\B(\Hi)_h$ of $U$ as
$$\psi_\xi(z)=\langle -iz\xi,\xi\rangle$$
for $z\in i\B(\Hi)_h$. Note that $\psi_\xi$ is a norming functional of $\log(\alpha_b)$, and since $\alpha$ is a short curve Theorem \ref{normingfunc} asserts that 
$$\psi_\xi(\alpha_t^{-1}\dot{\alpha}_t)=\langle -i\alpha_t^{-1}\dot{\alpha}_t\xi,\xi\rangle=\|\alpha_t^{-1}\dot{\alpha}_t\|$$
for all $t\in [a,b]$. Therefore the Cauchy-Schwartz inequality for $\xi$ and $ -i\alpha_t^{-1}\dot{\alpha}_t\xi$ is an equality, so there is a $\lambda\in\C$ such that 
$$ -i\alpha_t^{-1}\dot{\alpha}_t\xi=\lambda\xi$$
for all $t$. Hence
$$\langle \lambda\xi,\xi\rangle=\lambda=\|\alpha_t^{-1}\dot{\alpha}_t\|$$
for all $t$, and we conclude that $\xi\in\Hi_+(\alpha_t^{-1}\dot{\alpha}_t)$ for all $t\in [a,b]$. As in the first proof, Lemma \ref{lempspec} is used to prove that $\xi\in\Hi_+(\dot{\alpha}_t\alpha_t^{-1})$ for all $t\in [a,b]$.

\end{proof}

The next corollary of Theorem \ref{mainthm} is a generalization of Theorem \ref{ineqspec}, it can be proved by adapting the proof of this theorem and using induction. 

\begin{cor}
Let $u_1,\dots,u_n$ be unitary operators acting on the Hilbert space $\Hi$ such that $-1\notin\spec(u_j)$ for $j=1,\dots,n$, and such that $\sum_{j=1}^n\theta_+(u_j)<\pi$ and $\sum_{j=1}^n\theta_-(u_j)>-\pi$. Denote by $u=u_1\cdots u_n$ the product of these operators. Then $-1\notin\spec(u)$,
$$\sum_{j=1}^n\theta_+(u_j)\geq\theta_+(u)\quad\mbox{ and }\quad\sum_{j=1}^n\theta_-(u_j)\leq\theta_-(u).$$
If $\theta_+(u)=\sum_{j=1}^n\theta_+(u_j)<\pi$, then
$$\bigcap_{j=1}^n\Hi_+(u_j)=\Hi_+(u),$$
and if $\theta_-(u)=\sum_{j=1}^n\theta_-(u_j)>-\pi$, then
$$\bigcap_{j=1}^n\Hi_-(u_j)=\Hi_-(u).$$
\end{cor}

\begin{proof}
Let $x_1,\dots,x_n$ be skew-hermitian operators such that $e^{x_j}=u_j$ and $\|x_j\|<\pi$ for $j=1,\dots,n$, and let $\gamma:[0,n]\to U$ be a piecewise $C^1$ curve such that $\gamma_0=\id$ and $\gamma_t^{-1}\dot{\gamma}_t=x_j$ for $t\in (j-1,j)$ and $j=1,\dots,n$. Note that $\gamma_n=u$, $\Hi_+(u_j)=\Hi_+(x_j)$ and $\Hi_-(u_j)=\Hi_-(x_j)$ for $j=1,\dots,n$. It is easy to see that the assertions of the corollary are a special case of Theorem \ref{mainthm}.

\end{proof}

%\section*{Acknowledgements}
%We are grateful to an anonymous referee for his/her careful reading of the manuscript and valuable suggestions that improved the presentation of the article.

%\begin{bibdiv}
%\begin{biblist}

%\bib{kit}{article}{author={W. Audeh}, author={F. Kittaneh}, title={Singular value %inequalities for compact operators}, journal={Linear Algebra Appl.}, number={437}, %date={2012}, pages={no. 10, 2516--2522}}

%\bib{banach}{book}{author={S. Banach}, title={Th\'eorie des op\'erations lin\'eaires}, %series={Éditions Jacques Gabay, Sceaux}, date={1993}}

%\end{biblist}
%\end{bibdiv}

\noindent

\begin{thebibliography}{9}


\bibitem[AW98]{aw} S. Agnihotri, C. Woodward, \textit{Eigenvalues of products of unitary matrices and quantum Schubert calculus.} Math. Res. Lett. 5 (1998), no. 6, 817--836.

\bibitem[An14]{esteban} E. Andruchow, \textit{The Grassmann manifold of a Hilbert space.} Proceedings of the XIIth ``Dr. Antonio A. R. Monteiro'' Congress, 41--55, Actas Congr. ``Dr. Antonio A. R. Monteiro'', Univ. Nac. del Sur, Bahía Blanca, 2014. 

%\bibitem[ALV14]{thompson} J. Antezana, G. Larotonda, A. Varela, \textit{Optimal paths for symmetric actions in the unitary group.} Comm. Math. Phys. 328 (2014), no. 2, 481--497.


\bibitem[At87]{atkin1} C. J. Atkin, \textit{The Finsler geometry of groups of isometries of Hilbert space.} J. Austral. Math. Soc. Ser. A 42 (1987) no. 2, 196-222.

\bibitem[At89]{atkin2} C. J. Atkin. \textit{The Finsler geometry of certain covering groups of operator groups.} Hokkaido Math. J. 18 (1989) no. 1, 45-77.

\bibitem[CL11]{chau} H. F. Chau, Y. T. Lam, \textit{Elementary proofs of two theorems involving arguments of eigenvalues of a product of two unitary matrices.} J. Inequal. Appl. 2011, 2011:18, 6 pp.

\bibitem[Ch11]{chau0} H. F. Chau, \textit{Metrics on unitary matrices and their application to quantifying the degree of non-commutativity between unitary matrices.} Quantum Inf. Comput. 11 (2011), no. 9-10, 721--740. 

\bibitem[CPR00]{childs} A. Childs, J. Preskill, J. Renes, \textit{Quantum information and precision measurement.} Physics of quantum information. J. Modern Opt. 47 (2000), no. 2-3, 155--176.

\bibitem[HM90]{horn} R. A. Horn, R. Mathias, \textit{Cauchy-Schwarz inequalities associated with positive semidefinite matrices.} Linear Algebra Appl. 142 (1990), 63--82.

\bibitem[Ka78]{kato} T. Kato, \textit{Monotonicity theorems in scattering theory.} Hadronic J. 1 (1978), no. 1, 134--154.

\bibitem[La19]{larotonda} G. Larotonda, \textit{Metric geometry of infinite-dimensional Lie groups and their homogeneous spaces.} Forum Math. 31 (2019), no. 6, 1567--1605.



%\bibitem[Lar19]{lar19} G. Larotonda. \textit{Metric geometry of infinite-dimensional Lie groups and their homogeneous spaces}, Forum Math. 31 (2019), no. 6, 1567--1605.

\bibitem[NS58]{nudelman} A. A. Nudel'man, P. A. \v{S}varcman, \textit{The spectrum of the product of unitary matrices.} (Russian) Uspehi Mat. Nauk 13 1958 no. 6 (84), 111--117.

\bibitem[PR87]{portarecht} H. Porta, L. Recht, \textit{Minimality of geodesics in Grassmann manifolds}, Proc. Amer. Math. Soc. 100 (1987), 464-466.

\bibitem[Th86]{thompson} R. C. Thompson, \textit{Proof of a conjectured exponential formula.} Linear and Multilinear Algebra 19 (1986), no. 2, 187--197.

\bibitem[Ya92]{yaf} D. R. Yafaev, \textit{On the scattering matrix for perturbations of constant sign.} Ann. Inst. H. Poincar\'e Phys. Th\'eor. 57 (1992), no. 4, 361--384.







\end{thebibliography}
\end{document}